\documentclass[12pt]{article}
\usepackage{amsmath,amssymb,times,amsthm}
\usepackage{enumerate}

\newtheorem{thm}{Theorem}[section]

\newtheoremstyle{theorem}%name
  {10pt}          % space above
  {10pt}  % space below theoremnhb
  {\normalfont}  % body font
  {}     % ident - empty=no indent,  \parindent= paragraph indent
  {\bf}  % thm head font
  {. }    % punctuation after thm head
  { }    % space after thm head: `` ``=normal \newline=linebreak
  {}     % thm head specification

\theoremstyle{theorem}
\newtheorem{dfn}{Definition}[section]
\newtheorem{rem}{Remark}[section]
\newtheorem{cor}{Corollary}[section]

\begin{document}
\noindent \textit{Journal of the Kerala Statistical Association}, Vol. 7,  No. 1, July 1991, p. 01-07.

\noindent * {\normalfont This paper won the Young Scientist's Award at the 77$^{\rm th}$ session of the Indian Science Congress Association held at Cochin, January 1990.}

\begin{center}
{\bf *On geometric infinite divisibility, p-thinning and Cox processes} \\
\vspace{.4cm}
Sandhya E\\
Department of Statistics, University of Kerala\\
Thiruvananthapuram - 695 581, India.\\ 
email: \textit{esandhya@hotmail.com}
\end{center}

\begin{abstract}
The connections among geometric infinite divisibility, 
$p$-thinning and Cox processes are established in this paper. 
Some results on subordination of stochastic processes with 
stationary independent increments, connected with geometric 
infinite divisibility are derived. A characterization of the 
renewal process with semi-Mittag-Leffler as inter arrival time 
distribution is obtained in the context of $p$-thinning, which 
is an extension of a result due to Renyi (1956). It is identified 
that only geometrically infinitely divisible distributions can 
define a Cox and renewal process. An analogue of a theorem 
in Feller (1966) is given.
\end{abstract}

\section{Introduction}
The concept of geometric infinite divisibility (g.i.d) was
introduced by Klebanov \textit{et al.} (1984). 
A random variable (\textit{r.v.}) $X$ is said to be g.i.d if for every 
$p\in (0,1)$ it can be expressed as 
$X\overset{d}{=}\sum\limits^{N_p}_{j=1}$ $X_j^{(p)}$, where 
$N_p$, $X_1^{(p)}$, $X_2^{(p)}\ldots$ are independent, the
$X_j^{(p)}$ are \textit{i.i.d.} and $P\{N_p=n\}=p(1-p)^{n-1}$, $n =
1,2,\ldots$. Equivalently in terms of characteristic functions, 
definition is the following. Let $g(t)=E(e^{itX})$, 
$\phi_p(t)=E\left(e^{it\;X_1^{(p)}}\right)$. If for all
$p\in (0,1)$, $g(t)=\frac{p\phi_p(t)}{1-q\phi_p(t)}$,
$q=1-p$, then $X$ is said to be g.i.d.
They proved that a\textit{r.v.} with characteristic function $g(t)$ 
is g.i.d if and only if $g(t)=\frac{1}{1+\psi(t)}$, where
$e^{-\psi(t)}$ is infinitely divisible (i.d).
Pillai and Sandhya (1990) studied the class of g.i.d
distributions more deeply.

The idea of $p$-thinning goes back to Renyi (1956). 
Let $0 = t_0<t_1<\ldots <t_n$ denote the random epochs 
corresponding to a renewal process. If we retain every 
epoch $t_n$, $n=1,2,\ldots$ with probability `$p$' and delete it 
with probability $q =(1-p)$, independent of every other points 
and the process itself, the resulting process is
called the $p$-thinned process of the original process. 
Renyi used the term `rarification' for $p$-thinning. The process 
obtained by replacing each epoch $t_n$, $n=1,2,\ldots$ by `$pt_n$',
$p\in (0,1)$ is celled the contraction of the original process. 
Renyi (1956) proved that Poisson process is the only one that is invariant 
under contraction and $p$-thinning applied together.

An ordinary renewal process with inter arrival time distribution 
$G$ is said to be a Cox process if there exists a process with inter
arrival time distribution $F_p$ such that the process corresponding to 
$G$ in the $p$-thinning of the process corresponding to $F_p$, 
for all $p\in (0,1)$. That is, the process corresponding to $F_p$ is the $p$-inverse of the process corresponding to $G$ for all 
$p\in (0,1)$. For relevant work see Yannaros (1988, 1989).

Pillai (1990) introduced the class of distributions called the Mittag-Leffler 
distributions which are defined by their Laplace transform 
given by $\frac{1}{1+\lambda^\alpha}$, $0<a\leq 1$.
Obviously, $a=1$ corresponds to the exponential distribution. 
This is also a subclass of the semi-$\alpha$-Laplace distribution 
introduced in Pillai (1985). The semi-Mittag-Leffler distributions which we will be defining in the sequel contains the Mittag-Leffler distributions and is 
contained in the semi-$\alpha$-Laplace distributions.

In section 1, we present some results regarding the subordination 
of infinitely divisible processes (Feller (1966, p.335) directed by gamma, exponential and Mittag-Leffler processes such that
the increments of the subordinated process are g.i.d. The
main result in the next section is a characterization of the
renewal process with semi-Mittag-Leffler as the inter arrival time
distribution in the context of $p$-thinning. In the third section 
it is observed that only g.i.d. distributions can define a Cox and
renewal process. An analogue of a Theorem in
Feller (I966, p.294) is also established combining the main ideas of this paper.

\section{Geometric Infinite Divisibility and Subordination}

\begin{thm}\label{cox-thm-1.1} %Harem 1.1
Let $X$ be an i.d. r.v. with positive support. Let $X(t)$,
$t\geq 0$ be the associated process with stationary independent 
increments having Laplace transform $e^{-t\psi(\lambda)}$. Let $Y (t)$ be the process subordinated to $X(t)$ by the gamma operational time with distribution function
$ G_t(x) = \frac{1}{\Gamma(t)} \int\limits^x_0 y^{t-1} e^{-y} dy.$ Then the distribution of the increments of the subordinated
process is g.i.d. for $t\leq 1$.
\end{thm}

\begin{proof}
Let $F_s(x)$ be the distribution corresponding to the stochastic
process $X(s)$, $s\geq 0$ and $H_t(x)$ that of $Y(t)$, $t\geq 0$. 
By assumption 
$$
H_t(x) = \int\limits^\infty_0 
F_s(x) G_t \{ds\}.
$$
Taking Laplace transforms on both sides we get that of $H_t(x)$ as
\begin{align*}
\hat{H}_t(x) & = \int\limits^\infty_0 e^{-s\psi(\lambda)} G_t\{ds\}\\
& = \frac{1}{(1+\psi(\lambda))^t}\\
& = \frac{1}{1+(1+\psi(\lambda))^t -1}\\
& = \frac{1}{1+h(\lambda)}.
\end{align*}
%%%%%%

Since $\psi(\lambda)$ has complete monotone derivative and
$\psi(0)=0$, $h(\lambda)$ also has complete monotone derivative 
for $t\leq 1$ and $h (0)=0$ (Feller (1966, p.417). 
Now the result follows from Lemma 2.1 of Pillai and Sandhya (1990).
\end{proof}

\begin{thm} % flromn 1.2
Let $X$ be an i.d. r.v. with positive support. Let $X(t)$,
$t\geq 0$ be the associated process with stationary independent 
increments having Laplace transform $e^{-t\psi(\lambda)}$. Let 
$Y(t)$ be the process subordinated to $X(t)$ by the directing 
exponential operational time with distribution function 
$G_t(x)=\frac{1}{t}\int\limits^x_0$ $e^{-u/t}$ $du$. Then the distribution of the increments of the subordinated process 
is g.i.d for all $t>0$.
\end{thm}

\begin{thm} % Tltrwau 1.3
With the above set up if the operational time has a
Mittag-Leffler distribution with Laplace transform
$\frac{1}{1+\lambda^t}$, $0<t\leq 1$,
then the distribution of the increments of the subordinated process is g.i.d.
\end{thm}

The proof of the above two theorems follow along the same
lines as that of Theorem 2.1.

\section{$p$-thinning of Renewal Processes}
Let $X_1,X_2, \ldots $ denote a sequence of \textit{i.i.d. r.v}s with distribution function $F$. Then $S_n=X_1 + \ldots +X_n$, $n=1,2, \ldots $
denotes a renewal process. Here $X_1,X_2,\ldots$ are the inter 
arrival times of the renewal process. Then $X\overset{d}{=}$ 
$\sum\limits^{N_p}_{j=1}$ $X_j^{(p)}$ 
denotes the inter arrival time of the corresponding $p$-thinned 
process with thinning probability $q= (1-p)$. The inter arrival 
time distribution $G$ of $X$ is a geometric convolution of
$F$. i.e.,
$$
G(x)  = \sum^\infty_{n=1} pq^{n-1} F^{n*}(x),
$$
where $F^{n*}$ is the $n$-fold convolution of $F$.

Taking Laplace transforms on both sides of the above
equation we have
$$
g(\lambda) = \frac{p\phi_p(\lambda)}{1-q\phi_p(\lambda)}
$$
or
$$
\phi_p(\lambda) = \frac{g(\lambda)}{p+qg(\lambda)}.
$$
Thus a Laplace transform $g(\lambda)$ corresponds to a thinned 
renewal process if and only if there exists a Laplace transform
$\phi_p(\lambda)$ such that
$$
g(\lambda) = \frac{p\phi_p(\lambda)}{1-q\phi_p(\lambda)}
$$
or if
$$
\phi_p(\lambda) = \frac{g(\lambda)}{p+q(\lambda)}
$$
is a Laplace transform.

\begin{dfn} % Dcfidtbn 2.1
A distribution with positive support is said to be 
semi-Mittag-Leffler of exponent $\alpha$, $0<\alpha\leq 1$, if its 
Laplace transform is of the form $\frac{1}{1+\psi(\lambda)}$,
where $\psi(\lambda)$ satisfies $\psi(\lambda)=a\psi(b\lambda)$,
$0<b<a$ and $\alpha$ is the unique solution of $ab^{\alpha}$ $=1$.
\end{dfn}

\begin{thm} %% 2.1
The semi-Mittag-Leffler distribution is the only inter arrival
time distribution such that the corresponding renewal process 
is invariant under $p$-thinning (up to a scale change, 
$0<c<1$).
\end{thm}

\begin{proof}
Let $\phi(\lambda)$ be the Laplace transform of the inter arrival 
time distribution of some renewal process. Applying $p$-thinning 
and allowing a scale change $c, 0<c<1$, we get
%%%%
\begin{align*}
g(c\lambda) & = \phi(\lambda), \text{ where }
g(\lambda) = \frac{p\phi(\lambda)}{1-q\phi(\lambda)}\\
\text{i.e.,}\quad
\phi(\lambda) & = \frac{p\phi(c\lambda)}{1-q\phi(c\lambda)},\;
0<c<1.
\end{align*}
%%%%
Let $\phi(\lambda)=\frac{1}{1+\psi(\lambda)}$, where 
$\psi(\lambda)=\frac{1}{\phi(\lambda)}-1$ and then $\psi(\lambda)$
satisfies
%%%%%
\begin{equation}\label{cox-eq-2.1}
\psi(\lambda) = \frac{1}{p}\psi(c\lambda).
\end{equation}
%%%%%
This implies that the distribution is semi-Mittag-Leffler of exponent $\alpha$, $0<\alpha \leq 1$, choosing $c^{\alpha}=p$.

The converse follows by choosing $c^{\alpha}=p$ and retracing the steps.
\end{proof}

\begin{cor} % Corollary 2.1
If for two values of `$p$', say $p_1$ and $p_2$, 
$\frac{\log p_1}{\log p_2}$ is irrational, and (3.1) 
is satisfied, then $\psi(\lambda)=A\lambda^{\alpha}$, $0<\alpha\leq 1$, $A>0$ a constant, which implies that the distribution is Mittag-Leffler of exponent $\alpha$.
\end{cor}

Proof follows from Kagan, Linnik and Rao (1973, p.324). 

\section{Cox and Renewal Processes}
An ordinary renewal process with inter arrival time distribution 
$G$ is said to be a Cox process if there exists a process with 
inter arrival time distribution $F_p$ such that $G$ is the 
$p$-thinned process of $F_p$, for all $p\in (0,1)$. That is, $F_p$, 
is the $p$-inverse of $G$ for all $p\in (0,1)$. 
Yannaros (1989) proved that if a renewal process $N$ is a Cox 
process, than it cannot be the thinning of some non-renewal process,
and all the possible original processes are $p$-thinnings of other
renewal processes for every thinning parameter $p\in (0,1)$ and this
properly characterises the processes which are both Cox and renewal.

Stated in the form of Laplace transforms $g(\lambda)$ corresponds to
a Cox and renewal process if and only if 
$\phi_p(\lambda)=\frac{g(\lambda)}{p+qg(\lambda)}$
is a Laplace transform for all $p\in (0,1)$, as $\phi_p(\lambda)$ 
need not be a Laplace transform always. Yannaros (1988) 
has proved that this is possible if and only if 
$g(\lambda)=\frac{1}{1+\psi(\lambda)}$, $\psi(0)=0$ and
$\psi(\lambda)$ has a complete monotone derivative i.e., 
when $g(\lambda)$ is g.i.d. Therefore, we have,

\begin{thm}\label{cox-thm-3.1} % Theorem 3.1
An ordinary renewal process defines a Cox process if and only
if its inter arrival time distribution is g.i.d.
\end{thm}

Now we examine the behaviour of a $p$-thinned renewal process
with $p=1/n$ and $\phi_p(\lambda)=e^{-(1/n)\psi(\lambda)}$ as 
$n\rightarrow\infty$.

\begin{thm}\label{cox-thm-3.2} % Thnrenr 3.2
The $(1/n)$-thinning of an ordinary renewal process whose
inter arrival time distribution is i.d. with Laplace transform
$e^{-(1/n)\psi}$, as $n\rightarrow\infty$ defines a Cox and 
renewal process.
\end{thm}

\begin{proof}
The Laplace transform of the $(1/n)$ thinned process is given by
$$
\frac{(1/n) e^{-(1/n)\psi}}{1-(1-1/n) e^{-(1/n)\psi}}.
$$

Taking the limit as $n\rightarrow\infty$, the Laplace transform
reduces to $\frac{1}{1+\psi(\lambda)}$, which is g.i.d. Hence the
result.
\end{proof}

\begin{rem} %Remark 3.1
We saw that Laplace transforms of the form 
$\frac{1}{1+t\psi(\lambda)}$ for $t\geq 0$,
$\frac{1}{(1+\psi(\lambda))^t}$ and
$\frac{1}{1+\psi^t(\lambda)}$ for $t\leq 1$ are g.i.d.
This means that they can serve as the inter arrival times of Cox and 
renewal processes. Or in other words subordination of i.d. 
processes (with positive support) with exponential, gamma and 
Mittag-Leffler operational times generates Cox and renewal processes.
\end{rem}

An analogue of a theorem in Feller (1966, p.294) can be
proved combining the main ideas.

\begin{thm} % Tluomn 3.3
The following classes of probability distribution are identical.

\begin{enumerate}[(i)]
\item % (i) 
The set of all g.i.d. distributions with positive support and the limits of such distributions.
\item % (ii) 
Distributions with positive support whose characteristic function is
$$
g(t)=\lim_{n\rightarrow\infty}
\frac{1}{1+\sum\limits^{k_n}_{k=1}\lambda_{n,k}\left(1- e^{i\beta_{n,k}t}\right)}, \lambda_n>0
$$
\item %% (Iii) 
Limit distributions of $\sum\limits^{N_n}_{j=1}X_j$ where $N_n$ 
follows a geometric distribution with mean $n$, independent of 
$X_j$'s and $X_j$'s are uniformly asymptotically negligible \textit{r.v}s with a common distribution 
function having positive support and Laplace transform
$e^{-(1/n)\psi}$. 
\item % (iv) 
The set $T=\{F:F$ is the inter arrival time distribution of a
Cox and renewal process$\}$.
\end{enumerate}
\end{thm}

\begin{proof}\mbox{}

\begin{enumerate}[(i)]
\item % (l
Follows from Klebanov \textit{et al.} (1984).
\item % (Ii) 
Follows from Laha and Rohatgi(1979, p.237) and from the 
connection between i.d. and g.i.d. in terms of characteristic function 
(Klebanov \textit{et al.} (1984).
\item % (iii) 
Follows from Theorem \ref{cox-thm-3.2}.
\item % (iv) 
Follows from Theorem \ref{cox-thm-3.1}. \qedhere
\end{enumerate}
\end{proof}

\subsection*{Acknowledgement}
The author is indebted to Professor R. N. Pillai, University of
Kerala for his valuable guidance, corrections and suggestions. The
financial assistance from C.S.I.R, India, in the form of a Senior
Research Fellowship is also acknowledged, with thanks.

\subsection*{References}

Feller W (1966): \textit{An Introduction to Probability Theory and
Its Applications}, Vol. II, Wiley Eastern Limited, Calcutta.

Kagan A M; Linnik Yu V and Rao C R (1973):
\textit{Characterization Problems in Mathematical Statistics}, 
John Wiley and Sons, New York.

Klebanov L B; Maniya G M and Melamed I A (l984):
A Problem of Zolotarev and Analog of Infinitely Divisible and 
Stable Distributions in a Scheme for Summing a Random
Number of Random Variables, \textit{Theory of Probability and Its
Applications}, 29, 4, 757--760.

Laha R G and Rohatgi V K (1979): \textit{Probability Theory}, 
John Wiley and Sons, New York.

Pillai R N (1985): Semi-$a$-Laplace Distributions, Communications 
in statistics, \textit{Theory and Methods}, l4, 4, 991--1000.

Pillai R N (1990): On Mittag-Leffler Functions and Related
Distributions, \textit{Annals of the Institute of Statistical 
Mathematics}, 42, 157--161.

Pillai R N and Sandhya E (1990): On Geometric Infinite
Divisibility (To appear).

Renyi A (1956): \textit{A characterization of the Poisson 
Process}, Collected papers of Alfred Renyi, Vo1.1, 
Academic Press, New York.

Yannaros N (1988): On Cox Processes and Gamma Renewal 
Processes, \textit{Journal of Applied Probability}, 
25, 423--427.

Yannaros N (1989): On Cox and Renewal Processes, 
\textit{Statistics and Probability Letters}, 7, 431--433.

\end{document}